\documentclass{amsart}%
\usepackage{amssymb}
\usepackage{amsmath}
\usepackage{amsfonts}
\usepackage{graphicx}%
\newtheorem{theorem}{Theorem}
\theoremstyle{plain}

\newtheorem{corollary}{Corollary}

\newtheorem{lemma}{Lemma}

\newtheorem{proposition}{Proposition}

\numberwithin{equation}{section}
\begin{document}
\title[Fixed Point Approximation]{The Approximation to a Fixed Point}
\author{Sibylla Priess-Crampe}
\address{Mathematisches Institut der Universit\"{a}t M\"{u}nchen, Theresienstr.39,
80333 M\"{u}nchen}
\email{priess@math.lmu.de}
\author{Paulo Ribenboim}
\address{Department of Mathematics and Statistics, Queen's University, Kingston,
Ontario K7L3N6,Canada}
\email{mathstat@mast.queensu.ca}
\date{July 12, 2013}
\subjclass[2000]{Primary 54H25; Secondary 13A18.}
\keywords{ultrametric spaces, fixed points, process of approximation}

\begin{abstract}
In a spherically complete ultrametric space, a strictly contracting mapping
has a fixed point. We indicate in this paper how this fixed point can either
be reached or approximated.

\end{abstract}
\maketitle

\section{Introduction}

In our recent paper ``Ultrametric Dynamics'' \cite{UD} and earlier papers
\cite{Ba}, \cite{FP}, \cite{CP}, \cite{Att}, we proved theorems asserting the
existence of fixed points for self-maps of ultrametric spaces, as well as
similar theorems for common points of mappings. The proof of these theorems is
non-constructive, it gives no indication how to reach or to approximate a
fixed point by means of an algorithm. For the use in applications, the
indication of an algorithm of approximation is an essential complement to the
existence theorem.

Our treatment takes place in the following framework: $(X,d,\Gamma)$ is a
principally complete ultrametric space (see section 2 for the definitions of
the requested concepts). Let $\varphi:X\rightarrow X$ and let $Y$\ be a
non-empty subset of $X$. By the Fixed Point Theorem (see Theorem \ref{FPTh}),
if $\varphi$\ is a strictly contracting mapping, there exists a unique element
$z\in X$ such that $\varphi(z)=z$. For any $y\in Y$, we indicate an algorithm,
beginning at $y$, such that either the algorithm ``reaches'' $z$ (see definition
in section 3) or the algorithm provides an ``asymptotic approximation'' to
$z$\ (see definition in section 3).

A special case of our present Approximation Theorem was proved and used in our
papers ``Differential equations over valued fields (and more)'' \cite{DE} and
``Systems of differential equations over valued fields'' \cite{SDE}.

One of the merits of our paper is that we deal with ultrametric spaces having
sets of distances which are not necessarily totally ordered. For this reason
the technique of approximation is very involved, but the results will then be
applicable to general kinds of algorithms. Noticeable is that, when the fixed
point is not in the ultrametric space itself but in some completion, our
results indicate how to approximate the fixed point. This is often the
situation for differential equations.

\section{Definitions and relevant results}

Besides definitions, we present in this section also the results which are
required in the sequel. Proofs will only be given if we cannot refer to any
relevant publications.

\subsection{Ultrametric spaces, their extensions and completions}

Let $(\Gamma,\leq)$ be an ordered set with smallest element $0$. Let $X$ be a
non-empty set. A mapping $d:X\times X\rightarrow\Gamma$ is called an
\textit{ultrametric distance} (and $(X,d,\Gamma)$ an \textit{ultrametric
space}) if $d$ has the following properties for all $x,y,z\in X$ and
$\gamma\in\Gamma$:

\begin{description}
\item[(d1)] $d(x,y)=0$ if and only if $x=y$

\item[(d2)] $d(x,y)=d(y,x)$.

\item[(d3)] If $d(x,y)\leq\gamma$ and $d(yz)\leq\gamma$ then $d(x,z)\leq
\gamma$.
\end{description}

If there is no ambiguity, we simply write $X$\ instead of $(X,d,\Gamma)$.

If $\Gamma$\ is totally ordered, (d3) becomes:

\begin{description}
\item[(d3')] $d(x,z)\leq\max\{d(x,y),d(y,z)\}$ for all $x,y,z\in X$.
\end{description}

Let $(Y,d_{\mid Y},\Gamma_{Y})$ and $(X,d,\Gamma)$ be ultrametric spaces such
that $Y\subseteq X$ and $\Gamma_{Y}\subseteq\Gamma$. Assume that $\Gamma_{Y}$
has the induced order of $\Gamma$\ and the same $0$\ as $\Gamma$\ and that
furthermore, $d_{\mid Y}(Y\times Y)\subseteq\Gamma_{Y}$ and $d_{\mid
Y}(y,y^{\prime})=d(y,y^{\prime})$ for all $y,y^{\prime}\in Y$. Then
$(Y,d_{\mid Y},\Gamma_{Y})$\ is said to be a \textit{subspace} of
$(X,d,\Gamma)$\ and $X$\ is called an \textit{extension} of $Y$. Often we
simply write $d$\ instead of $d_{\mid Y}$.

The space $X$ is said to be \textit{solid\/} if for every $\gamma\in\Gamma$
and $x\in X$ there exists $y\in X$ such that $d(x,y)=\gamma$. If $X$ is solid
then $d(X\times X)=\Gamma$.

Let $\gamma\in\Gamma^{\bullet}=\Gamma\setminus\{0\}$ and $a\in X$. The set
$B_{\gamma}(a)=\{x\in X\mid d(a,x)\leq\gamma\}$ is called a \textit{ball}. The
element $a$\ is said to be a \textit{center\/} of $B_{\gamma}(a) $ and the
element $\gamma$ to be a \textit{radius\/} of $B_{\gamma}(a)$. If $x,y\in X$,
$x\neq y$, then $B(x,y)=B_{d(x,y)}(x)$ is called a \textit{principal ball}. If
$X$ is solid, every ball is principal.

In the following lemma, we list some properties of balls which can easily be verified.

\begin{lemma}
\label{balls}Let $\gamma,\delta\in\Gamma^{\bullet}$.

\begin{enumerate}
\item Let $x,y\in X$.

\begin{enumerate}
\item If $\gamma\leq\delta$ and $B_{\gamma}(x)\cap B_{\delta}(y)\neq\emptyset$
then $B_{\gamma}(x)\subseteq B_{\delta}(y)$.

\item If $B_{\delta}(y)\subset B_{\gamma}(x)$ then $\gamma\nleq\delta$.
\end{enumerate}

\item Concerning principal balls, if $x,y,z,u\in X$, $x\neq z$ and $y\neq u$, then:

\begin{enumerate}
\item $B(x,z)\subseteq B_{\delta}(y)$ if and only if $d(x,z)\leq\delta$ and
$x\in B_{\delta}(y)$.

\item If $B(x,z)\subset B_{\delta}(y)$ then $d(x,z)<\delta$.

\item If $B(x,z)=B(y,u)$ then $d(x,z)=d(y,u)$.
\end{enumerate}

\item Let $X$ be solid and $x,y\in X$.

\begin{enumerate}
\item $B_{\gamma}(x)\subseteq B_{\delta}(y)$ if and only if $\gamma\leq\delta$
and $x\in B_{\delta}(y)$.

\item If $B_{\gamma}(x)\subset B_{\delta}(y)$ then $\gamma<\delta$.

\item If $B_{\gamma}(x)=B_{\delta}(y)$ then $\gamma=\delta$.
\end{enumerate}

\item If $\Gamma$ is totally ordered and $B_{\delta}(y)\subset B_{\gamma}(x)$
then $\delta<\gamma$.
\end{enumerate}
\end{lemma}

A set of balls which is totally ordered by inclusion is said to be a
\textit{chain}. An ultrametric space $X$ is called \textit{spherically
complete} (respectively, \textit{principally complete}) if every chain of
balls of $X$ (respectively, principal balls of $X$) has a non-empty
intersection. Every spherically complete ultrametric space is principally
complete. The converse is true when $\Gamma$ is totally ordered or the space
is solid.

An ultrametric space $(X,d,\Gamma)$ is said to be \textit{complete} if every
chain of balls $\{B_{\gamma_{i}}(a_{i})\mid i\in I\}$, with $\inf\{\gamma
_{i}\mid i\in I\}=0$, has a non-empty intersection. Thus a spherically
complete ultrametric space is complete. If $\Gamma$\ is totally ordered and if
$\Gamma^{\bullet}$ does not have a smallest element, the ultrametric distance
induces on $X$\ a uniformity, hence also a topology. In this case, the concept
of completeness coincides with that given by the uniformity.

Assume that $\Gamma$\ is totally ordered. Let $(Y,d_{\mid Y},\Gamma)$\ be a
subspace of $(X,d,\Gamma)$ and assume that $d(Y\times Y)=d(X\times X)=\Gamma$.
If for every $x\in X$ and for every $y\in Y$, with $x\neq y$, there exists
$y^{\prime}\in Y$ such that $d(y^{\prime},x)<d(y,x)$, the extension $Y\prec X$
is called \textit{immediate}; if for every $x\in X$ and for every $0<\gamma
\in\Gamma$ there exists $y\in Y$ such that $d(x,y)<\gamma$, the subspace is
said to be \textit{dense} in $X$. If $\Gamma^{\bullet}$ does not have a
smallest element, this definition coincides with that given by the topology of
$X$. (We remark that both notions, ``immediate'' and ``dense'', can be defined
more generally for ultrametric spaces , where $\Gamma$\ is only ordered, see
\cite{PR97}).

\begin{theorem}
\label{compl}

\begin{enumerate}
\item Every ultrametric space $(X,d,\Gamma)$, with $\Gamma$\ totally ordered,
has an immediate extension which is spherically complete. (We call such an
extension a \textit{spherical completion} of $X$).

\item Every ultrametric space $(X,d,\Gamma)$, with $\Gamma$\ totally ordered,
has an extension $(X^{\prime},d,\Gamma)$ such that $X^{\prime}$\ is complete
and $X$ is dense in $X^{\prime}$. (We call such an extension a
\textit{completion} of $X$).

\item Let $(Y,d,\Gamma_{Y})$\ be a subspace of $(X,d,\Gamma)$. Assume that
$\Gamma$\ is totally ordered and that $\Gamma_{Y}^{\bullet}$ is coinitial in
$\Gamma^{\bullet}$ and that furthermore $d(Y\times Y)=\Gamma_{Y}^{\bullet}$,
$d(X\times X)=\Gamma$. If $X$\ is complete then there exists one and only one
completion $\widehat{Y}$ of $Y$\ which is a subspace of $X$.
\end{enumerate}
\end{theorem}

\begin{proof}
1), 2): see \cite{PR97} 7.9, 8.11 or \cite{Scho}.

3): Let $\mathcal{S}$ be the set of all ultrametric subspaces $S$\ of
$X$\ such that $Y$\ is dense in $S$. The set $\mathcal{S}$\ is ordered by
inclusion. Let $\{S_{i}\mid i\in I\}$ be a totally ordered subset of
$\mathcal{S}$. Then $S=\underset{i\in I}{%
{\textstyle\bigcup}
}S_{i}$ is a subspace of $X$\ and $Y$ is dense in $S$. Thus $S\in\mathcal{S}$
is an upper bound for all $S_{i}$, $i\in I$. By Zorn's Lemma, there exists a
maximal element in $\mathcal{S}$\ which we denote again by $S$.

We show that $S$\ is complete. Since $\Gamma_{Y}^{\bullet}$ is coinitial in
$\Gamma^{\bullet}$ and $\Gamma_{Y}^{\bullet}=\Gamma_{S}^{\bullet}=d(S\times
S)\setminus\{0\}$, we have that $\Gamma_{S}^{\bullet}$ is coinitial in
$\Gamma^{\bullet}$. Therefore, a set $\Delta\subseteq d(S\times S)\setminus
\{0\}$ has in $\Gamma_{S}^{\bullet}$ the infimum $0$ if and only if the
infimum of $\Delta$ in $\Gamma^{\bullet}$\ is $0$, thus we may just write
$\inf\Delta=0$. We assume that $S$\ is not complete. Then there exists a chain
$\{B_{\gamma_{i}}^{S}(a_{i})\mid i\in I\}$ of balls in $S$\ with $\inf
\{\gamma_{i}\mid i\in I\}=0$ and $%
{\textstyle\bigcap}
B_{\gamma_{i}}^{S}(a_{i})=\emptyset$. Since $X$\ is complete and for each
$i\in I$, $B_{\gamma_{i}}^{S}(a_{i})=S\cap B_{\gamma_{i}}(a_{i})$, where
$B_{\gamma_{i}}(a_{i})$ denotes the ball with center $a_{i}$ and radius
$\gamma_{i}$ in $X$, there exists $z\in X$ such that $\{z\}=%
{\textstyle\bigcap}
B_{\gamma_{i}}(a_{i})$. Let $S^{\prime}=S\cup\{z\}$. Then $S^{\prime}$\ is a
subspace of $X$\ which properly contains $S$, so also $Y$. To prove that
$Y$\ is dense in $S^{\prime}$, it suffices to show that if $0<\gamma\in\Gamma$
there exists $y\in Y$ such that $d(y,z)<\gamma$. Since $\inf\{\gamma_{i}\mid
i\in I\}=0$ there exists $\gamma_{i}$\ with $0<\gamma_{i}<\gamma$. From $Y$
dense in $S$\ and $a_{i}\in S$, it follows that there exists $y\in Y$ such
that $d(y,a_{i})<\gamma_{i}$. Since moreover $z\in B_{\gamma_{i}}(a_{i})$ then
$d(z,y)\leq\max\{d(z,a_{i}),d(y,a_{i})\}\leq\gamma_{i}<\gamma$. Thus $Y$ is
dense in $S^{\prime}$. So $S^{\prime}\in\mathcal{S}$, which contradicts the
maximality of $S$\ in $\mathcal{S}$. We have proved that $S$\ is complete,
hence a completion of $Y$\ in $X$.

It remains to show that $Y$\ has at most one completion in $X$. Assume that
$\widehat{Y}_{1}$, $\widehat{Y}_{2}$ are completions of $Y$\ in $X$. Let
$\widehat{y}_{1}\in\widehat{Y}_{1}$. For each $\gamma\in\Gamma^{\bullet}$
there exists $y_{\gamma}\in Y$ such that $d(\widehat{y}_{1},y_{\gamma}%
)<\gamma$. Since $\widehat{Y}_{2}$\ is complete and $\inf\{\gamma\mid\gamma
\in\Gamma^{\bullet}\}=0$ there exists $\widehat{y}_{2}\in\widehat{Y}_{2}$ with
$\{\widehat{y}_{2}\}=\underset{\gamma\in\Gamma^{\bullet}}{%
{\textstyle\bigcap}
}B_{\gamma}(y_{\gamma})$. Hence $\widehat{y}_{1}=\widehat{y}_{2}$, which shows
that $\widehat{Y}_{1}\subseteq\widehat{Y}_{2}$. By the same argumentation, we
conclude that $\widehat{Y}_{2}\subseteq\widehat{Y}_{1}$, thus $\widehat{Y}%
_{1}=\widehat{Y}_{2}$.
\end{proof}

\textit{For the rest of this subsection, we assume that }$\Gamma$\textit{\ is
totally ordered.}

Let $(X,d,\Gamma)$ be an ultrametric space. Let $\lambda$ be a limit ordinal,
let $\xi=(x_{\iota})_{\iota<\lambda}$ be a family of elements of $X$. We say
that $\xi$ is a \textit{Cauchy family\/} if for every $\gamma\in
\Gamma^{\bullet}$ there exists $\iota_{0}=\iota_{0}(\gamma,\xi)<\lambda$ such
that if $\iota_{0}\leq\iota<\kappa<\lambda$ then $d(x_{\iota},x_{\kappa
})<\gamma$.

The family $\xi=(x_{\iota})_{\iota<\lambda}$ is said to be
\textit{pseudo-convergent\/} if there exists $\iota_{0}=\iota_{0}(\xi
)<\lambda$ such that if $\iota_{0}\leq\iota<\kappa<\mu<\lambda$ then
$d(x_{\kappa},x_{\mu})<d(x_{\iota},x_{\kappa})$. We note that if
$\xi=(x_{\iota})_{\iota<\lambda}$ is pseudo-convergent, the elements
$x_{\iota}$\thinspace, for $\iota_{0}(\xi)\leq\iota<\lambda$, are all distinct
and if $\iota_{0}(\xi)\leq\iota<\kappa<\mu<\lambda$ then $d(x_{\iota
},x_{\kappa})=d(x_{\iota},x_{\mu})$; this element is denoted by $\xi_{\iota}%
$\thinspace. Hence if $\iota_{0}\leq\iota<\kappa<\lambda$ then $\xi_{\iota
}>\xi_{\kappa}$\thinspace.

The element $y\in X$ is a \textit{limit\/} of the Cauchy family $\xi$ if for
every $\gamma\in\Gamma^{\bullet}$ there exists $\iota_{1}=\iota_{1}%
(\gamma)<\lambda$, such that if $\iota_{1}\leq\iota<\lambda$ then
$d(y,x_{\iota})<\gamma$. A Cauchy family has at most one limit. Indeed, if
$y$, $z$ are limits, then $d(y,z)<\gamma$ for all $\gamma\in\Gamma^{\bullet}$,
so $y=z$. The ultrametric space $X$\ is complete if and only if every Cauchy
family has a limit in $X$.

The element $y\in X$ is a \textit{pseudo-limit\/} of the pseudo-convergent
family $\xi=(x_{\iota})_{\iota<\lambda}$ if there exists $\iota_{1}=\iota
_{1}(\xi,y)$, $\iota_{0}(\xi)\leq\iota_{1}<\lambda$, such that if $\iota
_{1}\leq\iota<\lambda$ then $d(y,x_{\iota})\leq\xi_{\iota}$\thinspace. If $y$
is a pseudo-limit of $\xi$ then $z\in X$ is a pseudo-limit of $\xi$ if and
only if $d(y,z)<\xi_{\iota}$ for all $\iota$ such that $\iota_{1}\leq
\iota<\lambda$. The ultrametric space $X$ is spherically complete if and only
if every pseudo-convergent family of $X$\ has a pseudolimit in $X$ (see
\cite{Ba}).

\subsection{Contracting mappings}

Let $(X,d,\Gamma)$ be an ultrametric space. A mapping $\varphi:X\rightarrow X$
is said to be \textit{strictly contracting}, if for all $x,x^{\prime}\in X$,
with $x\neq x^{\prime}$, $d(\varphi(x),\varphi(x^{\prime}))<d(x,x^{\prime})$.
An element $z\in X$ with $\varphi(z)=z$ is called a \textit{fixed point} of
$\varphi$. In \cite{UD} (see also \cite{Ba}, \cite{FP}), we proved the
following Fixed Point Theorem:

\begin{theorem}
\label{FPTh}Assume that $(X,d,\Gamma)$ is principally complete and that
$\varphi:X\rightarrow X$ is strictly contracting. Then $\varphi$ has exactly
one fixed point $z\in X$.
\end{theorem}

If the ultrametric space $(Y,d,\Gamma)$ is not spherically complete and
$\varphi:Y\rightarrow Y$ is strictly contracting the following result
guarantees an appropriate extension of $\varphi$ for the case that $\Gamma
$\ is totally ordered.

\begin{theorem}
\label{mapext}Assume that $\Gamma$\ is totally ordered and that $(X,d,\Gamma)$
is spherically complete. Let $Y$\ be a subspace of $X$. If $\psi:Y\rightarrow
Y$ is strictly contracting, there exists $\varphi:X\rightarrow X$ such that
$\varphi$\ is strictly contracting and extends $\psi$. If moreover, $d(Y\times
Y)\setminus\{0\}$ is coinitial in $d(X\times X)\setminus\{0\}$, the
restriction $\varphi_{\mid\widehat{Y}}$ of $\varphi$\ to the completion
$\widehat{Y}$ of $Y$\ in $X$\ is uniquely determined.
\end{theorem}

\begin{proof}
The existence of $\varphi$\ is proved in \cite{P05}.

We show that $\varphi$\ is uniquely determined on $\widehat{Y}$. Let
$\varphi_{1}$, $\varphi_{2}$ be extensions of $\psi$\ to strictly contracting
mappings from $X$ to $X$, and assume that $\varphi_{1\mid\widehat{Y}}%
\neq\varphi_{2\mid\widehat{Y}}$. Then there exists $\widehat{y}\in\widehat
{Y}\setminus Y$ such that $\varphi_{1}(\widehat{y})\neq\varphi_{2}(\widehat
{y})$. So $0<\gamma=d(\varphi_{1}(\widehat{y}),\varphi_{2}(\widehat{y}))$.
Since $Y$\ is dense in $\widehat{Y}$\ there exists $y\in Y$ such that
$d(y,\widehat{y})<\gamma$. Hence $d(\psi(y),\varphi_{1}(\widehat
{y}))=d(\varphi_{1}(y),\varphi_{1}(\widehat{y}))<d(y,\widehat{y})$ and
$d(\psi(y),\varphi_{2}(\widehat{y}))=d(\varphi_{2}(y),\varphi_{2}(\widehat
{y}))<d(y,\widehat{y})$, thus $\gamma=d(\varphi_{1}(\widehat{y}),\varphi
_{2}(\widehat{y}))\leq\max\{d(\psi(y),\varphi_{1}(\widehat{y})),d(\psi
(y),\varphi_{2}(\widehat{y}))\}<d(y,\widehat{y})<\gamma$, which is a
contradiction. Therefore $\varphi_{1\mid\widehat{Y}}=\varphi_{2\mid\widehat
{Y}}$.
\end{proof}

\section{The Process of Approximation}

Our purpose is now to indicate how to reach or to approximate a fixed point.

Let $(X,d,\Gamma)$\ be a principally complete ultrametric space. We shall
assume that $\Gamma^{\bullet}$\ does not have a smallest element. To exclude
the trivial case, we also assume that $X$\ has at least two elements.

Let $\varphi:X\rightarrow X$ be a strictly contracting mapping, so by Theorem
\ref{FPTh}, $\varphi$ has a unique fixed point, which we denote by $z$.

\subsection{The Approximation Theorem}

We shall work with families of elements of $X$. If $\lambda$ is an ordinal
number, let $l(\lambda)$ denote the set of ordinal numbers $\mu<\lambda$. As
it is known, $\lambda$ may be identified with $l(\lambda)$\ and, by
definition, the cardinal of $\lambda$ is $card\;\lambda=card\;l(\lambda)$. Let
$\kappa$ be a limit ordinal with $card\;\kappa>card\;\Gamma$. For every
ordinal $\lambda$ such that $\lambda<\kappa$, let $\mathcal{P}_{\lambda}$ be
the set of all families $\alpha=(a_{\iota})_{\iota<\lambda}\in X^{l(\lambda)}$
which satisfy the following conditions:

\begin{description}
\item[i)] if $\iota+1<\lambda$ then $a_{\iota+1}=\varphi(a_{\iota})\neq
a_{\iota}$;

\item[ii)] $(d(a_{\iota},a_{\iota+1}))_{\iota+1<\lambda}$ is strictly decreasing;

\item[iii)] if $\mu$ is a limit ordinal, $\mu<\lambda$, then $d(a_{\mu
},a_{\iota})\leq d(a_{\iota},a_{\iota+1})$ for all $\iota<\mu$.
\end{description}

If $\lambda=1$, $\mathcal{P}_{1}$\ is naturally identified with $X$, so
$\mathcal{P}_{1}\neq\emptyset$. Let $\mathcal{P}$ be the union of the sets
$\mathcal{P}_{\lambda}$\ for $\lambda<\kappa$. If $y\in X$ let $\mathcal{P}%
_{y}$ be the set of families in $\mathcal{P}$\ with $a_{0}=y$.

We say that $\alpha=(a_{\iota})_{\iota<\lambda}$ \textit{reaches} $z$\ if
there exists $\iota_{0}<\lambda$ such that $a_{\iota_{0}}=z$. It follows that
$\varphi(a_{\iota_{0}})=z=a_{\iota_{0}}$, hence by (i) $\iota_{0}+1=\lambda$.
Thus $\lambda$ is not a limit ordinal.

Let $\lambda$ be a limit ordinal, let $\alpha=(a_{\iota})_{\iota<\lambda}%
\in\mathcal{P}$\ and for every $\iota<\lambda$ let $B_{\iota}=B_{\iota}%
(\alpha)=B_{d(a_{\iota},\varphi(a_{\iota}))}(a_{\iota})$, so by (i),
$B_{\iota}$\ is a principal ball. By Lemma \ref{balls}, we have $B_{\iota
+1}\subseteq B_{\iota}$ and, in fact, $B_{\iota+1}\subset B_{\iota}$, because
$d(a_{\iota},a_{\iota+1})>d(a_{\iota+1},a_{\iota+2})$, so $a_{\iota}\notin
B_{\iota+1}$. For every limit ordinal $\mu\leq\lambda$ let $I_{\mu}%
(\alpha)=\underset{\iota<\mu}{%
{\textstyle\bigcap}
}B_{\iota}(\alpha)$. Since $X$\ is principally complete then $I_{\mu}%
(\alpha)\neq\emptyset$. We say that $\alpha=(a_{\iota})_{\iota<\lambda}%
\in\mathcal{P}$\ is an \textit{asymptotic approximation} to $z$ (or more
simply, an \textit{approximation} to $z$) if $\lambda$ is a limit ordinal and
$I_{\lambda}(\alpha)=\underset{\iota<\lambda}{%
{\textstyle\bigcap}
}B_{\iota}(\alpha)=\{z\}$. We note that an approximation to $z$\ does not
reach $z$, because $\lambda$\ is a limit ordinal.\medskip

\medskip The next result will be called the \textit{Approximation Theorem}.

\begin{theorem}
\label{Appr}\textit{Let $X$ be principally complete, let $Y\subseteq X$,
$Y\neq\emptyset$. Assume that $z$ cannot be reached by any }$\alpha
\in\mathcal{P}$\textit{\ such that }$a_{0}\in Y\setminus\{z\}$\textit{. Then
for every }$y\in Y\setminus\{z\}$\textit{\ there exists an asymptotic
approximation }$\alpha=(a_{\iota})_{\iota<\lambda}$\textit{\ to }%
$z$\textit{\ such that }$a_{0}=y$\textit{.}
\end{theorem}

\begin{proof}
The proof requires some preliminary considerations about the set $\mathcal{P}
$.

1) The order relation on $\mathcal{P}$.

Let $\alpha=(a_{\iota})_{\iota<\lambda}$\ and $\alpha^{\prime}=(a_{\iota
}^{\prime})_{\iota<\lambda}$\ be families in $\mathcal{P}$. We define
$\alpha\leq\alpha^{\prime}$ when $\lambda\leq\lambda^{\prime}$ and $a_{\iota
}^{\prime}=a_{\iota}$ for all $\iota<\lambda$. It is immediate to verify that
$\leq$\ is an order relation. Moreover, for every $\lambda$, the order
restricted to $\mathcal{P}_{\lambda}$\ is trivial.

2) Let $y\in Y$, $y\neq z$. The ordered set $\mathcal{P}_{y}$\ is inductive.

Let $C$\ be a non-empty set, for every $c\in C$ let $\alpha^{c}=(a_{\iota}%
^{c})_{\iota<\lambda_{c}}\in\mathcal{P}_{y}$, assume that if $c\neq c^{\prime
}$ then $\alpha^{c}\neq\alpha^{c^{\prime}}$ and that the set $A=\{\alpha
^{c}\mid c\in C\}$ is a totally ordered subset of $\mathcal{P}_{y} $. It
follows that if $\alpha^{c}<\alpha^{c^{\prime}}$\ then $\lambda_{c}%
<\lambda_{c^{\prime}}$. We recall that from $\alpha^{c}\in\mathcal{P}_{y}$ it
follows that $\lambda_{c}<\kappa$. We consider two cases.

a) $L=\{\lambda_{c}\mid c\in C\}$ has a largest element $\lambda_{c_{1}}$.
Then $\alpha^{c}\leq\alpha^{c_{1}}$ for every $c\in C$, otherwise there exists
$c_{2}\in C$ such that $\alpha^{c_{1}}<\alpha^{c_{2}}$, hence $\lambda_{c_{1}%
}<\lambda_{c_{2}}$, which is a contradiction. In this case, $\alpha^{c_{1}}%
$\ is an upper bound for $A$.

b) $L$ does not have a largest element. Since $\lambda_{c}<\kappa$ for every
$c\in C$, there exists the smallest element $\mu$ such that $\lambda_{c}<\mu$
for every $c\in C$. So $\mu\leq\kappa$.

If $\mu=\nu+1$ then by the minimality of $\mu$, there exists $c_{1}\in C$ such
that $\nu\leq\lambda_{c_{1}}$ and therefore $\nu=\lambda_{c_{1}}$, because
$\mu>\lambda_{c_{1}}$. In this case, $\lambda_{c_{1}}$ is the largest element
in $L$, which has been excluded.

We have shown that $\mu$\ is a limit ordinal. Now if $\iota$\ is an ordinal
such that $\iota<\mu$, by the minimality of $\mu$\ there exists $\lambda
_{c}\in L$ such that $\iota\leq\lambda_{c}<\mu$. Since $L$\ does not have a
largest element, there exists $c^{\ast}\in C$ such that $\iota\leq\lambda
_{c}<\lambda_{c^{\ast}}<\mu$. We define $\widetilde{a_{\iota}}=a_{\iota
}^{c^{\ast}}$. It is immediate to verify that $\widetilde{a_{\iota}}$\ is
well-defined, independently of the choice of $c^{\ast}\in C $ such that
$\lambda_{c}<\lambda_{c^{\ast}}<\mu$. By (ii), the family $(d(a_{\iota
}^{c^{\ast}},a_{\iota+1}^{c^{\ast}}))_{\iota<\mu}$ of elements of $\Gamma$\ is
strictly decreasing, hence all these elements are pairwise distinct. So
$card\;\mu\leq\Gamma$. Since $card\;\Gamma<card\;\kappa$ thus $\mu<\kappa$,
which implies that $\widetilde{\alpha}=(\widetilde{a_{\iota}})_{\iota<\mu}%
$\ belongs to $\mathcal{P}$. Furthermore, $\alpha^{c}<\widetilde{\alpha}$ for
every $c\in C$. Hence $\widetilde{\alpha}$\ is an upper bound for $A$. This
concludes the proof that $\mathcal{P}_{y}$\ is inductive.

By Zorn's Lemma, there exists a maximal $\alpha\in\mathcal{P}_{y}$. That is,
for every $y\in Y\setminus\{z\}$ there exists a maximal $\alpha\in\mathcal{P}$
such that $a_{0}=y$.

3) Proof of the theorem.

We assume that $z$\ is not reached by any family in $\mathcal{P}_{y}$\ for
every $y\in Y\setminus\{z\}$. By (2), for every $y\in Y\setminus\{z\}$ there
exists a maximal $\alpha=(a_{\iota})_{\iota<\lambda}\in\mathcal{P} $ such that
$a_{0}=y$. First we observe that $\lambda$\ is a limit ordinal. We assume the
contrary, let $\lambda=\iota_{0}+1$. Since z\ is not reached by $\alpha$\ then
$a_{\iota_{0}}\neq z$, so $a_{\iota_{0}}\neq\varphi(a_{\iota_{0}})$, hence
$d(\varphi(a_{\iota_{0}}),\varphi^{2}(a_{\iota_{0}}))<d(a_{\iota_{0}}%
,\varphi(a_{\iota_{0}}))$. Let $\alpha^{\prime}=(a_{\iota}^{\prime}%
)_{\iota<\lambda+1}$ where $a_{\iota}^{\prime}=a_{\iota}$ for all
$\iota<\lambda$ and $a_{\lambda}^{\prime}=\varphi(a_{\iota_{0}})$. So
$\alpha^{\prime}\in\mathcal{P}$, $\alpha<\alpha^{\prime}$. This is impossible,
because $\alpha$\ is maximal in $\mathcal{P}$. Thus, as stated, $\lambda$ is a
limit ordinal.

Since $X$ is principally complete, and each $B_{\iota}(\alpha)$ is a principal
ball of $X$, then $I_{\lambda}(\alpha)=\underset{\iota<\lambda}{%
{\textstyle\bigcap}
}B_{\iota}(\alpha)\neq\emptyset$. We show that $I_{\lambda}(\alpha)=\{z\}$.
Let $t\in I_{\lambda}(\alpha)$. We note that $t\neq a_{\iota}$ for all
$\iota<\lambda$. Indeed, if there exists $\iota_{0}<\lambda$ such that
$t=a_{\iota_{0}}$, then $t\notin B_{\iota_{0}+1}$ which is a contradiction.
Now we show that $\varphi(t)\in I_{\lambda}(\alpha)$. We have $d(\varphi
(t),a_{\iota+1})=d(\varphi(t),\varphi(a_{\iota}))<d(t,a_{\iota})\leq
d(a_{\iota},a_{\iota+1})$ for all $\iota<\lambda$. It follows that
$d(t,\varphi(t))\leq d(a_{\iota},a_{\iota+1})$ for every $\iota<\lambda$.
Hence $d(t,\varphi(t))<d(a_{\iota},a_{\iota+1})$ for every $\iota<\lambda$.
Let $\alpha^{\prime}=(a_{\iota}^{\prime})_{\iota<\lambda+1}$ be defined by
$a_{\iota}^{\prime}=a_{\iota}$ for all $\iota<\lambda$ and $a_{\lambda
}^{\prime}=t$. So $\alpha^{\prime}\in\mathcal{P}$, because $d(t,a_{\iota})\leq
d(a_{\iota},a_{\iota+1})$ for every $\iota<\lambda$. We have $\alpha
<\alpha^{\prime}$, which is contrary to the maximality of $\alpha$. This shows
that $t=\varphi(t)$, so $t=z$ and we deduce that $I_{\lambda}(\alpha)=\{z\}$.
Hence $\alpha$\ is an asymptotic approximation to $z$.
\end{proof}

\smallskip

Under the assumptions of the Approximation Theorem, if $y\in Y\setminus\{z\} $
there exists the smallest limit ordinal $\lambda$\ for which there exists an
approximation $\alpha=(a_{\iota})_{\iota<\lambda}$ to $z$\ such that $a_{0}%
=y$. So the set $\mathcal{M}=\{\alpha=(a_{\iota})_{\iota<\lambda}\mid\alpha$
is an approximation to $z$ and $a_{0}=y\}$ is not empty. For every $\alpha
\in\mathcal{M}$ and each limit ordinal $\mu<\lambda$ the set $I_{\mu}(\alpha
)$\ contains properly $z$. The set $I_{\mu}(\alpha)$\ may be considered to be
a measure of the accuracy of the approximation $\alpha$, when restricted to
$\alpha_{\mid\mu}=(a_{\iota})_{\iota<\mu}$.

\begin{corollary}
\label{CorA}Let $X$\ be principally complete. \textit{If }$y\in X$\textit{,
}$y\neq z$\textit{\ , then either there exists }$\alpha\in\mathcal{P}_{y}%
$\textit{\ which reaches }$z$\textit{, or if this is not the case, there
exists an approximation }$\alpha\in\mathcal{P}_{y}$ \textit{to }$z$\textit{.}
\end{corollary}

\begin{proof}
The corollary is a special case of the Approximation Theorem, taking $Y=\{y\}
$, where $y\neq z$.
\end{proof}

The proof of the Approximation Theorem suggests the method to reach or to
approximate the fixed point.

Let $y\in Y$. If $y=z$ there is nothing to do. If $y\neq z$ let $a_{0}=y$ and
$a_{1}=\varphi(a_{0})\neq a_{0}$. If $a_{1}=z$ then $z$\ has been reached by
the family consisting only of $a_{0}$, $a_{1}$. If $a_{1}\neq z$ let
$a_{2}=\varphi(a_{1})\neq a_{1}$. The procedure may be iterated. It may happen
that there exists $n_{0}>2$ such that $a_{n_{0}}=z$, so $z$\ has been reached.
Or, for every $n<\omega$, $a_{n}\neq z$. Let $\alpha=(a_{n})_{n<\omega}$. If
the set $I_{\omega}(\alpha)$\ consists of only one element, this element is
the fixed point $z$. If $I_{\omega}(\alpha)$\ has more than one element, we
may choose any one of the elements of $I_{\omega}(\alpha)$\ and call it
$a_{\omega}$. Then $a_{\omega+1}=\varphi(a_{\omega})$ if $\varphi(a_{\omega
})\neq a_{\omega}$, $a_{\omega+2}=\varphi(a_{\omega+1})$ if $\varphi
(a_{\omega+1})\neq a_{\omega+1}$, etc.. It may happen that there exists
$n\geq0$ such that $a_{\omega+n}=z$, or one needs to consider $I_{2\omega
}(\alpha^{\prime})$, where $\alpha^{\prime}=(a_{\iota}^{\prime})_{\iota
<2\omega}$, with $a_{\iota}^{\prime}=a_{\iota}$ for $\iota<\omega$ and
$a_{\iota}^{\prime}$, defined as indicated for $\omega\leq\iota<2\omega$. Even
though there exists a family $\alpha\in\mathcal{P}$\ which reaches or
approximates $z$, in general it is not possible to predict what will happen,
in particular, when the algorithm will stop.

\subsection{\textbf{The case when $\Gamma$ is totally ordered}}

Henceforth we shall assume that $\Gamma$\ is totally ordered and that
$\Gamma^{\bullet}$\ does not have a smallest element.

We shall use the following notations:

$\mathcal{A}$ = set of all approximations $\alpha$\ to $z$,

$\mathcal{PC}$ = set of all pseudo-convergent families in $X$.

\begin{proposition}
\label{apc}\textit{We have:}

\begin{enumerate}
\item $\mathcal{A}\subseteq\mathcal{PC}$\textit{.}

\item \textit{Let }$\mathcal{PLA}$\textit{\ be the set of all pseudo-limits of
all }$\alpha\in\mathcal{A}$\textit{. Then }$\mathcal{PLA}=\{z\}$\textit{.}

\item \textit{If }$\alpha\in\mathcal{A}$\textit{, if }$\alpha^{\prime}%
\in\mathcal{P}$\textit{\ and }$\alpha<\alpha^{\prime}$\textit{\ then }%
$\alpha^{\prime}$\textit{\ reaches }$z$\textit{.}
\end{enumerate}
\end{proposition}

\begin{proof}
1): We show that $\alpha$\ is a pseudo-convergent family in $X$. Since
$\alpha$\ is an approximation to $z$\ then $\lambda$\ is a limit ordinal. We
shall prove that if $\iota<\mu<\nu<\lambda$ then $d(a_{\iota},a_{\mu
})>d(a_{\mu},a_{\nu})$. For this purpose, we prove that $d(a_{\iota}%
,a_{\iota+1})=d(a_{\iota},a_{\mu})$. The proof is by induction on $\mu$. It is
trivial if $\mu=\iota+1$.

Let $\iota+1<\mu$. We consider two cases:

i) $\mu=\kappa+1$.

By induction, $d(a_{\iota},a_{\iota+1})=d(a_{\iota},a_{\kappa})$, since
$a_{\iota}\neq a_{\kappa}$ then $d(a_{\iota},a_{\kappa})>d(a_{\iota+1},a_{\mu
})$ and therefore $d(a_{\iota},a_{\mu})=d(a_{\iota},a_{\iota+1})$.

ii) $\mu$ is a limit ordinal.

By construction of $\alpha$, we have $d(a_{\mu},a_{\iota})\leq d(a_{\iota
},a_{\iota+1})$ for all $\iota<\lambda$. If $d(a_{\mu},a_{\iota})<d(a_{\iota
},a_{\iota+1})$ for some $\iota<\lambda$ then $d(a_{\iota},a_{\iota
+1})=d(a_{\mu},a_{\iota+1})$. From $a_{\mu}\neq a_{\iota}$ then $d(a_{\mu
+1},a_{\iota+1})<d(a_{\mu},a_{\iota})<d(a_{\mu},a_{\iota+1})$, hence
$d(a_{\iota},a_{\iota+1})=d(a_{\mu},a_{\mu+1})$. This is absurd, so
$d(a_{\iota},a_{\iota+1})=d(a_{\mu},a_{\iota})$ for all $\iota<\mu$. This
concludes the proof by induction.

In a similar way $d(a_{\mu},a_{\nu})=d(a_{\mu},a_{\mu+1})$ for $\mu
<\nu<\lambda$. It follows that if $\iota<\mu<\nu<\lambda$ then $d(a_{\iota
},a_{\mu})=d(a_{\iota},a_{\iota+1})>d(a_{\mu},a_{\mu+1})=d(a_{\mu},a_{\nu})$.
So we have proved that $\alpha$\ is a pseudo-convergent family in $X$.

2): Assume that $\alpha\in\mathcal{A}$\ then $I_{\lambda}(\alpha)=\{z\}$, so
$d(z,a_{\iota})\leq d(a_{\iota},a_{\iota+1})=d(a_{\iota},a_{\mu})$ for all
$\mu$\ such that $\iota<\mu<\lambda$. Thus $z$\ is a pseudo-limit of the
pseudo-convergent family $\alpha$.

Let $t\in X$, $t\neq z$, then $t\notin I_{\lambda}(\alpha)$. So there exists
$\iota_{0}<\lambda$ such that $d(t,a_{\iota_{0}})\nleq d(a_{\iota_{0}%
},a_{\iota_{0}+1})$, that is $t\notin B_{\iota_{0}}(\alpha)$. Hence for every
$\iota$\ such that $\iota_{0}<\iota<\lambda$, we also have $t\notin B_{\iota
}(\alpha)$, that is $d(t,a_{\iota})\nleq d(a_{\iota},a_{\iota+1})=d(a_{\iota
},a_{\mu})$ for $\iota<\mu<\lambda$. So $t$\ is not a pseudo-limit of $\alpha$.

3): Let $\alpha^{\prime}\in\mathcal{P}$ be such that $\alpha<\alpha^{\prime}$.
Since $\alpha^{\prime}\in\mathcal{P}$\ we have for every $\iota<\lambda$,
$d(a_{\lambda}^{\prime},a_{\iota}^{\prime})\leq d(a_{\iota}^{\prime}%
,a_{\iota+1}^{\prime})$ or equivalently, $d(a_{\lambda}^{\prime},a_{\iota
})\leq d(a_{\iota},a_{\iota+1})$ because $a_{\iota}^{\prime}=a_{\iota}$,
$a_{\iota+1}^{\prime}=a_{\iota+1}$. Hence $a_{\lambda}^{\prime}\in I_{\lambda
}(\alpha)=\{z\}$. So $\alpha^{\prime}$ reaches $z$.
\end{proof}

Let $\alpha=(a_{\iota})_{\iota<\lambda}\in\mathcal{P}$, let $\Sigma_{\alpha
}=\{d(a_{\iota},\varphi(a_{\iota}))\mid\iota<\lambda\}$. We note that
$0\in\Sigma_{\alpha}$ if and only if $\alpha$\ reaches $z$\ and, in this case,
$\lambda$\ is not a limit ordinal. Let $\Lambda_{\varphi}=\{d(x,\varphi
(x))\mid x\in X,x\neq z\}$. Then $\Sigma_{\alpha}\setminus\{0\}\subseteq
\Lambda_{\varphi}\subseteq\Gamma^{\bullet}$.

Let $(Y,d,\Gamma)$ be a subspace of $(X,d,\Gamma)$. If $\varphi$\ is such that
$\varphi(Y)\subseteq Y$, let $\Lambda_{\varphi}^{Y}=\{d(y,\varphi(y))\mid y\in
Y,y\neq z\}$. Since $X$\ is principally complete (and $\Gamma$\ totally
ordered), $X$ is spherically complete. If moreover, $d(Y\times Y)\setminus
\{0\}$ is coinitial in $d(X\times X)\setminus\{0\}$ then by Theorem
\ref{compl}, $Y$ has one and exactly one completion $\widehat{Y}$ in $X$.

\begin{proposition}
\label{coin}\textit{Let }$\alpha$ \textit{be an approximation to }$z$\textit{.
Then we have:}

\begin{enumerate}
\item $\Sigma_{\alpha}$ \textit{is coinitial in }$\Lambda_{\varphi}$.

\item \textit{Assume that }$(Y,d,\Gamma)$\textit{\ is a subspace of
}$(X,d,\Gamma)$\textit{\ and that }$\varphi(Y)\subseteq Y$\textit{. Assume
moreover, that }$d(Y\times Y)\setminus\{0\}$\textit{\ is coinitial in
}$d(X\times X)\setminus\{0\}$\textit{. If }$z\in\widehat{Y}\setminus
Y$\textit{\ then }$\Lambda_{\varphi}^{Y}$\textit{\ is coinitial in }%
$\Lambda_{\varphi}$\textit{.}

\item \textit{If }$\Sigma_{\alpha}$\textit{\ is coinitial in }$\Gamma
^{\bullet}$\textit{\ then }$\alpha$\textit{\ is a Cauchy family and }%
$z=\lim\alpha$\textit{.}

\item \textit{If }$X$\textit{\ is solid then }$\Lambda_{\varphi}%
=\Gamma^{\bullet}$\textit{, furthermore, }$\alpha$\textit{\ is a Cauchy family
and }$z=\lim\alpha$\textit{.}
\end{enumerate}
\end{proposition}

\begin{proof}
1): Assume that $\Sigma_{\alpha}$\ is not coinitial in $\Lambda_{\varphi} $.
So there exists $x\in X$, $x\neq z$, such that $d(x,\varphi(x))<d(a_{\iota
},a_{\iota+1})$ for all $\iota<\lambda$. Let $\alpha^{\prime}=(a_{\iota
}^{\prime})_{\iota<\lambda+1}$ be defined by $a_{\iota}^{\prime}=a_{\iota}$
for all $\iota<\lambda$ and $a_{\lambda}^{\prime}=x$. Then $\alpha^{\prime}%
\in\mathcal{P}$ and $\alpha<\alpha^{\prime}$. By Proposition \ref{apc}, part
(3), $\alpha^{\prime}$\ reaches $z$, while $\alpha$\ does not reach $z$. So
$x=a_{\lambda}^{\prime}=z$, and this is absurd.

2): Since $z\in\widehat{Y}\setminus Y$ there exists a limit ordinal $\rho
$\ and a Cauchy family $(y_{\nu})_{\nu<\rho}$, with $y_{\nu}\in Y$, such that
$z=\underset{\nu<\rho}{\lim}\,y_{\nu}$. Let $d(x,\varphi(x))\in\Lambda
_{\varphi}$. Since $\underset{\nu<\rho}{\lim}\,y_{\nu}=z$, there exists
$\nu<\rho$ such that $d(y_{\nu},z)\leq d(x,\varphi(x))$. Thus $d(\varphi
(y_{\nu}),\varphi(z))=d(\varphi(y_{\nu}),z)<d(y_{\nu},z)\leq d(x,\varphi(x))$,
which implies that $d(y_{\nu},\varphi(y_{\nu}))=d(y_{\nu},z)\leq
d(x,\varphi(x))$. Hence $\Lambda_{\varphi}^{Y}$\ is coinitial in
$\Lambda_{\varphi}$.

3): Let $\gamma\in\Gamma^{\bullet}$, by assumption there exists $\iota
_{0}<\lambda$ such that $d(a_{\iota_{0}},\varphi(a_{\iota_{0}}))=d(a_{\iota
_{0}},a_{\iota_{0}+1})\leq\gamma$. By Proposition \ref{apc}, $\alpha$ is
pseudo-convergent. Hence $d(a_{\iota},a_{\mu})<d(a_{\iota_{0}},a_{\iota_{0}%
+1})\leq\gamma$ for all $\iota$, $\mu$ such that $\iota_{0}<\iota<\mu<\lambda
$. By assumption, $z\in I_{\lambda}(\alpha)$, so $d(z,a_{\iota})\leq
d(a_{\iota},\varphi(a_{\iota}))<\gamma$ for every $\iota$\ such that
$\iota_{0}<\iota<\lambda$. This shows that $\alpha$\ is a Cauchy family and
$z=\lim\alpha$.

4): Let $0<\gamma\in\Gamma$. Since $X$\ is solid, there exists $x\in X$ such
that $d(x,z)=\gamma$. So $x\neq z$, hence $d(z,\varphi(x))=d(\varphi
(z),\varphi(x))<d(z,x)$, which implies that $d(x,\varphi(x))=d(z,x)=\gamma$.
Thus $\Lambda_{\varphi}=\Gamma^{\bullet}$. By (1), $\Sigma_{\alpha}$ is
coinitial in $\Lambda_{\varphi}=\Gamma^{\bullet}$. Hence by (3), $\alpha$ is a
Cauchy family in $X$ and $z=\lim\alpha$.
\end{proof}

In the next theorem, we shall study the following situation:

$(Y,d,\Gamma)$ is an ultrametric space, the mapping $\psi:Y\rightarrow Y$ is
strictly contracting, and the spherically complete ultrametric space
$(X,d,\Gamma)$ is an extension of $Y$, furthermore, we assume that $d(Y\times
Y)\setminus\{0\}$ is coinitial in $d(X\times X)\setminus\{0\}$. (For example,
$X$ could be the spherical completion of $Y$, see Theorem \ref{compl}). By
Theorem \ref{compl}, $Y$ has exactly one completion $\widehat{Y} $ in $X$, and
by Theorem \ref{mapext}, $\psi$ has an extension to a strictly contracting
mapping $\varphi$\ from $X$ to $X$. The mapping $\varphi$\ is not uniquely
determined by $\psi$, however its restriction to the completion $\widehat{Y}%
$\ of $Y$\ in $X$ is uniquely determined (see Theorem \ref{mapext}). In
general, different extensions of $\psi$\ to strictly contracting mappings of
$X$ will lead to different fixed points of these mappings. But if
$z\in\widehat{Y}$ then, since all these extensions coincide on $\widehat{Y}$,
$z$\ is the fixed point of all these mappings.

\begin{theorem}
\label{alg}Let $Y$, $X$ and the mappings $\psi$, $\varphi$ be as desribed
above. \textit{Assume that }$\alpha=(a_{\iota^{\prime}})_{\iota^{\prime
}<\lambda^{\prime}}$\textit{, with }$a_{0}\in Y$\textit{, is (with respect to
}$\varphi$\textit{) an approximation to }$z\in X\setminus Y$\textit{\ and that
furthermore, }$\Sigma_{\alpha}$\textit{\ is coinitial in }$\Gamma^{\bullet}%
$\textit{. Then }$z\in\widehat{Y}$ \textit{and there exists an approximation
}$\beta=(b_{\iota})_{\iota<\lambda}$\textit{\ to }$z$\textit{\ such that
}$b_{0}=a_{0}$\textit{\ and }$b_{\iota}\in Y$\textit{\ for all }$\iota
<\lambda$.
\end{theorem}

\begin{proof}
By Proposition \ref{coin}, $\alpha$ is a Cauchy family and $z=\lim\alpha
\in\widehat{Y}$.

We now refer to the proof of Theorem \ref{Appr}.

Let $\kappa$, $\mathcal{P}\ $and the order relation on $\mathcal{P}\ $be as
described there. Let $\mathcal{T}$ be the set of all $\beta=(b_{\iota}%
)_{\iota<\lambda}$\ of $\mathcal{P}$\ such that $b_{0}=a_{0}$, $b_{\iota}\in
Y$ for every $\iota<\lambda$ and $z\in\underset{\iota<\lambda}{%
{\textstyle\bigcap}
}B_{\iota}$, where $B_{\iota}=B_{d(b_{\iota},\psi(b_{\iota}))}(b_{\iota})$.
(We note that $b_{\iota}\neq z$ for every $\iota<\lambda$, because
$z\in\widehat{Y}\setminus Y$).

First we show that $\mathcal{T}$, with the restriction of the order of
$\mathcal{P}$, is inductive. Let $C$\ be a non-empty set, for every $c\in C$
let $\beta^{c}=(b_{\iota}^{c})_{\iota<\lambda_{c}}\in\mathcal{T}$. Assume that
$\beta^{c}\neq\beta^{c^{\prime}}$, if $c\neq c^{\prime}$, and that
$\mathcal{B}=\{\beta^{c}\mid c\in C\}$ is totally ordered. If $L=\{\lambda
_{c}\mid c\in C\}$ has a largest element $\lambda_{c_{1}}$, it follows, as
shown in the proof of Theorem \ref{Appr}, that $\beta^{c_{1}}$\ is an upper
bound for $\mathcal{B}$.

Thus there remains the case that $L$ does not have a largest element. We
conclude as in part (b) of the proof of Theorem \ref{Appr} that there exists
the smallest ordinal $\mu$ such that $\lambda_{c}<\mu$ for every $c\in C$,
that $\mu\leq\kappa$ and that $\mu$\ is a limit ordinal. Now we define in
similar way, as explained there, a family $\widetilde{\beta}=(\widetilde
{b_{\iota}})_{\iota<\mu}$\ which belongs to $\mathcal{P}$\ and which
furthermore has the following properties: $\widetilde{b_{0}}=a_{0}$,
$\widetilde{b_{\iota}}\in Y$ for every $\iota<\mu$\ and $z\in\underset
{\iota<\mu}{%
{\textstyle\bigcap}
}\widetilde{B_{\iota}}$, with $\widetilde{B_{\iota}}=B_{d(\widetilde{b}%
_{\iota},\psi(\widetilde{b}_{\iota}))}(\widetilde{b}_{\iota})$. Thus
$\widetilde{\beta}\in\mathcal{T}$ is an upper bound for $\mathcal{B}$. Hence
$\mathcal{T}$\ is inductive.

Moreover, $\mathcal{T}\neq\emptyset$, because $(a_{\iota^{\prime}}%
)_{\iota^{\prime}<\omega_{0}}\in\mathcal{T}$. Thus by Zorn's Lemma,
$\mathcal{T}$\ has a maximal element $\beta=(b_{\iota})_{\iota<\lambda}$. Then
$\lambda$\ is a limit ordinal. Indeed, if not, let $\lambda=\iota_{0}+1$.
Since $b_{\iota_{0}}\in Y$, also $\psi(b_{\iota_{0}})\in Y$, so $b_{\iota_{0}%
}\neq z$, $\psi(b_{\iota_{0}})\neq z$ and $b_{\iota_{0}}\neq\psi(b_{\iota_{0}%
})$. Therefore $d(z,\psi^{2}(b_{\iota_{0}}))<d(z,\psi(b_{\iota_{0}%
}))<d(z,b_{\iota_{0}})$, hence $d(\psi(b_{\iota_{0}}),\psi^{2}(b_{\iota_{0}%
}))=d(z,\psi(b_{\iota_{0}}))<d(z,b_{\iota_{0}})=d(b_{\iota_{0}},\psi
(b_{\iota_{0}}))$. Thus if $b_{\iota}^{\ast}=b_{\iota}$ for $\iota<\lambda$
and $b_{\lambda}^{\ast}=\psi(b_{\iota_{0}})$ then $\beta<\beta^{\ast
}=(b_{\iota}^{\ast})_{\iota<\lambda+1}$, furthermore $z\in B_{d(b_{\lambda
}^{\ast},\psi(b_{\lambda}^{\ast}))}(b_{\lambda}^{\ast})$, so $\beta^{\ast}%
\in\mathcal{T}$ contrary to the maximality of $\beta$\ in $\mathcal{T}$.
Hence, $\lambda$\ is a limit ordinal. Since $z\in\underset{\iota<\lambda}{%
{\textstyle\bigcap}
}B_{\iota}(\beta)$, we have $z\in I_{\lambda}(\beta)$. Assume there exists
$t\in X$ such that $t\neq z$ and $t\in I_{\lambda}(\beta)$. Then $0<d(t,z)$.
Since $z\in\widehat{Y}\setminus Y$, there exists a Cauchy family $(y_{\nu
})_{\nu<\rho}$ in $Y$, $\rho$ a limit ordinal, such that $z=\underset{\nu
<\rho}{\lim}\,y_{\nu}$. Thus there exists $\nu_{0}<\rho$ such that
$d(z,y_{\nu_{0}})\leq d(t,z)$. Then $d(\psi(y_{\nu_{0}}),z)=d(\varphi
(y_{\nu_{0}}),z)<d(y_{\nu_{0}},z)\leq d(t,z)$. So $d(\psi(y_{\nu_{0}}%
),y_{\nu_{0}})=d(y_{\nu_{0}},z)\leq d(t,z)$. It follows that if $b_{\iota
}^{\prime}=b_{\iota}$ for $\iota<\lambda$ and $b_{\lambda}^{\prime}=y_{\nu
_{0}}$ then $\beta^{\prime}=(b_{\iota}^{\prime})_{\iota<\lambda+1}>\beta$ and
moreover, $\beta^{\prime}\in\mathcal{T}$, because $z\in B_{\lambda
}=B_{d(y_{\nu_{0}},\psi(y_{\nu_{0}}))}(y_{\nu_{0}})$. This contradicts the
maximality of $\beta$\ in $\mathcal{T}$. Hence $I_{\lambda}(\beta)=\{z\}$.
\end{proof}

In our papers \cite{DE} and \cite{SDE}, we give some applications of the
results of this section to provide solutions or approximations to solutions of
twisted polynomial equations and of polynomial differential equations.

\bigskip


\begin{thebibliography}{99}                                                                                               %


\bibitem {Ba}Priess-Crampe, S., Der Banachsche Fixpunktsatz f\"{u}r
ultrametrische R\"{a}ume, Result. Math. 18 (1990), 178-186.

\bibitem {P05}Priess-Crampe, S., Remarks on some theorems of functional
analysis, Contemporary Mathematics 384 (2005), 235-246.

\bibitem {FP}Priess-Crampe, S.and Ribenboim, P., Fixed Points, combs and
generalized power series, Abh. Math. Sem. Univ. Hamburg 63 (1993), 227-244.

\bibitem {PR97}Priess-Crampe, S. and Ribenboim, P.,\ Generalized ultrametric
spaces II, Abh. Math. Sem. Univ. Hamburg 67 (1997), 19-31.

\bibitem {CP}Priess-Crampe, S.and Ribenboim, P., The common point theorem for
ultrametric spaces, Geom. Ded. 72 (1998), 105-110.

\bibitem {Att}Priess-Crampe, S.and Ribenboim, P., Fixed point and attractor
theorems for ultrametric spaces, Forum Math. 12 (2000), 53-64.

\bibitem {SDE}Priess-Crampe, S. and Ribenboim, P.,\ Systems of differential
equations over valued fields, Contemp. Math. 319 (2003), 299-318.

\bibitem {DE}Priess-Crampe, S. and Ribenboim, P., Differential equations over
valued fields (and more), \ J. Reine Angew. Math. 576 (2004), 123-147.

\bibitem {UD}Priess-Crampe, S. and Ribenboim, P.,\ Ultrametric dynamics, Ill.
J. 55 (2011), 287-303.

\bibitem {Scho}Sch\"{o}rner, E., On immediate extensions of ultrametric
spaces, Result. Math. 29 (1996), 361-370.
\end{thebibliography}

\end{document}